\documentclass[10pt]{article}
\usepackage{amssymb,amsthm,amsmath}
\usepackage{hyperref}
\usepackage{graphicx}
\usepackage{url}
\usepackage{amsmath}

\newtheorem{theorem}{Theorem}
\newtheorem{definition}{Definition}
\newtheorem{proposition}{Proposition}

\newtheorem{remark}{Remark}
\newtheorem{example}{Example}
\newtheorem{corollary}{Corollary}

\renewenvironment{matrix}[1][r]{%
  \hskip -\arraycolsep\array{*\c@MaxMatrixCols {#1}}%
}{%
  \endarray \hskip -\arraycolsep
}

\renewenvironment{bmatrix}[1][r]{%
  \left[\begin{matrix}[#1]%
}{%
  \end{matrix}\right]
}

\mathchardef\ordinarycolon\mathcode`\:
\mathcode`\:=\string"8000
\begingroup \catcode`\:=\active
  \gdef:{\mathrel{\mathop\ordinarycolon}}
\endgroup

\newcommand{\F}{\mathbb{F}}
\newcommand{\FF}{\mathcal{F}}

\begin{document}
\title{Construction of Arithmetic Secret Sharing Schemes by Using  Torsion Limits}

\author{ {\normalsize  Seher Tutdere \footnote{Department of Mathematics, Gebze Technical University, Turkey (stutdere@gmail.com) This paper was presented at the conference Arithmetic, Geometry, Cryptography and Coding Theory (AGCT-15.)},
and Osmanbey Uzunkol \footnote{Mathematical and Computational Sciences Labs, T\"UB\.ITAK B\.ILGEM,  Turkey (osmanbey.uzunkol@tubitak.gov.tr)} 
}}
\date{}
\maketitle
\begin{abstract}

Recent results of Cascudo, Cramer, and Xing on the construction of arithmetic secret sharing schemes are improved by using some new bounds on the torsion limits of algebraic function fields. Furthermore, new bounds on the torsion limits of certain towers of function fields are given.\\   


\end{abstract}

\section{Introduction}

Secret sharing is a cryptographic mechanism allowing to distribute shares among different parties. This is achieved by a trusted dealer in such a way that only authorized subset of parties can determine the secret \cite{beimel2}. Unlike conventional cryptographic sc{\large }hemes, secret sharing schemes enable the user to eliminate the root of trust problem \cite{beimel2,shamir}. Furthermore, secret sharing has plenty of privacy preserving real-life applications ranging from access controls \cite{naor}, oblivous transfers \cite{tassa} to biometric authentication schemes \cite{ign:will}.

If the authorized subset has the cardinality larger than a predetermined lower bound, then  secret sharing schemes have the property of {\it threshold access structure} \cite{farras}. Moreover, a secret sharing scheme is called {\it ideal} if the shares have the same size as secrets \cite{beimel2}. Shamir's secret sharing scheme is a classical example of an ideal secret sharing scheme having threshold access structure. Since the shares are computed and reconstructed by using only linear algebra \cite{karnin:1983}, it is also an example of {\it linear secret sharing schemes} (LSSS).  Ito et al. \cite{ito} introduced secret sharing schemes for general access structures. Moreover, an LSSS can be constructed for any access structure \cite{ito2}. However, the shares grow exponentially in the number of parties, and the optimization of secret sharing schemes for arbitrary access structures is a difficult problem \cite{beimel2}.

Chen and Cramer \cite{chen:cram} introduced an LSSS defined over a finite field using algebraic-geometry codes (AG-codes). Unlike the general case, this scheme has the advantage that shares are much smaller than the number of parties since one uses algebraic curves with many rational points. Therefore, this achieves larger information rate by generalizing Shamir's secret sharing scheme into an algebra-geometric setting. One inevitable disadvantage (due to the bounds on MDC \cite{chen:cram}) is that this scheme is an ideal {\it ramp secret sharing scheme, i.e. a quasi-threshold scheme}. In particular, one has the property that the scheme has $t$-rejecting and $t+1+2g$-accepting structure, where $g$ is the genus of the underlying maximal algebraic curve.

Cascudo, Cramer, and Xing \cite{cas:cram:xing} introduced \textit{arithmetic secret sharing schemes} which are special quasi-threshold $\mathbb{F}_q$-linear secret sharing schemes based on AG-codes. They can be used as the main algorithmic primitives in realizing information theoretically secure multi-party computation schemes (in particular, communication-efficient two-party cryptography) and verifiable secret sharing schemes \cite{chaum:crepeau:damgaard,cramer:damgaard:maurer}. 
More precisely, it is shown in \cite{chen:cram} that \textit{asymptotically good} arithmetic secret sharing schemes can be used to achieve constant-rate communication in secure two-party communication by removing logarithmic terms which appears if one instead uses Shamir's secret sharing scheme \cite{shamir}. As argued in \cite{cas:cram:xing}, these schemes can be also used as an important primitive in plenty of other useful applications in cryptography including zero-knowledge for circuit satisfiability \cite{ishai:2007} and efficient oblivous transfers \cite{ishai:2008}.

Constructing asymptotically good arithmetic secret sharing schemes is based on some special families of algebraic function fields. Besides the well-known notion of \textit{Ihara limits} for constructing asymptotically good function field towers, the notion \textit{torsion limits}  for algebraic function fields is introduced in \cite{cas:cram:xing}.  Geometrically, in order to construct arithmetic secret sharing schemes with asymptotically good properties, we need not only to have algebraic curves with many rational points but also to have jacobians (of corresponding algebraic curves) having comparably small $d$-torsion subgroups. On the algebraic side, the torsion limit for a function field tower with a given Ihara limit gives information on the size of $d-$torsion subgroups of the corresponding degree-zero divisor class groups. In \cite{cas:cram:xing}, the authors  give  asymptotical results improving the classical bounds of Weil \cite{weil} on the size of torsion subgroups of abelian varieties over finite fields. For this purpose, the existence of solutions for certain Riemann-Roch systems of equations is investigated. The authors further give new bounds on the  torsion limits of certain families of function fields. Consequently, they use these bounds in constructing asymptotically good arithmetic secret-sharing schemes by weakening the lower bound condition on the Ihara constant.

 In this work, we made some modifications and improvements on their results by using the bound on class number given by \cite{la:mar}. Moreover, we estimated the torsion limit of an important class of towers of function fields introduced by Bassa
et al. depending on the Ihara limit given in \cite{bas:be:gar:stic}. For example for the case $d>2$, these new bounds can easily be adapted to improve the communication complexity of zero knowledge protocols  for multiplicative relations introduced in \cite{cramer:damgaard:pastro}.

In Section \ref{prem} we revisit the preliminaries about algebraic function fields together with algebraic-geometry codes and Riemann-Roch systems of equations. We further investigate the bounds on the torsion limits in Section \ref{prem}. Then, we apply the result for the bounds on the torsion limits for  function field towers in Section \ref{torsion}. In Section \ref{secret} new conditions for the construction of arithmetic schemes are investifated and the results are proven. We construct families of arithmetic secret sharing schemes with uniformity in Section \ref{uniformity}. Moreover, we give examples yielding to infinite families of arithmetic secret sharing schemes in Section \ref{uniformity}. Finally, Section \ref{conclusion} concludes the paper.

\section{Preliminaries}\label{prem}

Let $F/\mathbb{F}_q$ be a function field over  the finite field $\mathbb{F}_q$  with $q$ elements, where $q$ is a power of a prime number $p$.  We denote by $g:=g(F)$ its genus, by $B_i(F)$ 
its number of places of degree $i$ for any $i\in \mathbb{N}$, and by $\mathbb{P}(F)$ its set of rational places. 

An asymptotically exact sequence of algebraic  function fields $\mathcal{F}={F_i}_{i\geq 0}$ over a finite field $\F_q$ is a sequence of function fields  such that for all $m\geq 1$ the following limit exists:
\[\beta_m(\FF)=\lim_{i\to \infty} \frac{B_m(F_i)}{g_i}.\]
It is well-known that any tower of function fields over any  finite  field is an exact sequence, see for instance \cite{hess:stic:tut}. 
 
We will use the following notations frequently:
\begin{itemize}
	\item $A_n$: The number of effective divisors of degree $n$, for $n\geq 1$. 
	\item $h_i$: The class number of $F_i/\F_q$ for any family of function fields $\FF=(F_i)_{i\geq 1}$.
	\item $\mathbb{P}^{(k)}(F)$: The set of places of $F/\F_q$ having degree $k\in \mathbb{N}$.
	\item  $\log:=\ln $.
    \item 	$CI(F):=\textrm{Div}(F)/\textrm{Prin}(F)$: The divisor class  group of $F/\F_q$.
     \item 	$CI_s(F):=\{[D]\; : \deg D=s \}$, where $[D] \in CI(F)$ stands for the divisor class containing $D$.
	\item  Div$^0(F)$: The group of divisors of $F$  with degree zero,
	\item  $\mathcal{J}_F=\textrm{Div}^0(F)/\textrm{Prin}(F)$:  The zero divisor class group of $F$ with cardinality $|\mathcal{J}_F|=h(F)$, which is called the \textit{class number}.
	\item $C(D,G)_{L}:$ The image of the map $\phi: \mathcal{L}(G)\rightarrow \mathbb{F}_q^k\times\mathbb{F}_q^n$,\\ $f\mapsto (f(Q_1), \cdots f(Q_k),f(P_1),\cdots f(P_n))$, where $\mathcal{L}(G)$ is the Riemann-Roch space of $G$, $k,n\in \mathbb{N},\  n\geq k$, $G$ is a divisor of $F$, $Q_1,\cdots,Q_k, P_1,\cdots P_n\in \mathbb{P}^{(1)}(F)$ are pairwise distinct $\mathbb{F}_q$-places with $D=\sum_{j=1}^{k}Q_j+\sum_{i=1}^{n}P_i$ and $\mbox{supp }D\cap \mbox{supp }G=\emptyset$.
\end{itemize}
For a positive integer $r$, let 
\[\mathcal{J}_F[r]:=\{[D] \in \mathcal{J}_F: r\cdot [D]=\mathcal{O}\}\]
 be the $r$-torsion  subgroup of $\mathcal{J}_F$, where $\mathcal{O}$ denotes the identity element of $\mathcal{J}_F$.
For each family  $\mathcal{F}=\{F/\mathbb{F}_q\}$ of function fields with $g(F)\to \infty$, the limit
\[J_r(\mathcal{F}):=\liminf_{F\in \mathcal{F}}\frac{\log_q |\mathcal{J}_F[r]|}{g(F)}\]
is called the \textit{$r$-torsion limit} of the family $\mathcal{F}$. Let $a\in \mathbb{R}$ and $\mathfrak{F}$ be the set of families $\{\mathcal{F}\}$ of 
function fields over $\mathbb{F}_q$ such that in each family genus tends to infinity and the Ihara limit 
\begin{equation*}
A(\mathcal{F})=\lim_{g(F)\to \infty} \frac{B_1(F)}{g(F)}\geq a \textrm{ for every }  \mathcal{F} \in \mathfrak{F}.
\end{equation*}
Then the asymptotic quantity $J_r(q,a)$ is defined by 
\[J_r(q,a):=\liminf_{\mathcal{F} \in \mathfrak{F}} J_r(\mathcal{F}).\]
We note that  we only consider the Ihara limit for function field families $\mathcal{F}$ for which this limit exists following the lines of \cite[Remark 2.1]{cas:cram:xing}. 

   An $(n,t,d,r)$-arithmetic secret sharing scheme for $\mathbb{F}_q^k$ over $\mathbb{F}_q$ is an $n$-code $C$ for $\mathbb{F}_q^k$ such that
	$t\geq 1$, $d\geq 2$, $C$ is $t$-disconnected, the $d$ powering $C^{*d}$ is an $n$-code for $\mathbb{F}_q^k$, and  $C^{*d}$   is $r$-reconstructing. For further details, the relation of these codes with $C(D,G)_{L}$, and the concept of uniformity we refer to \cite[pp. 3873-3875]{cas:cram:xing}.

Firstly, we  investigate  the bounds on torsion limits in the following theorem by combining the bounds in Theorems 2.3 and 2.4 of \cite{cas:cram:xing}:
\begin{theorem}\label{thm_bounds}  Let $\mathbb{F}_q$ be a finite field of characteristic $p$.  For any integer $r\geq 2$,  set $J_r:=J_r(q,A(q))$. Write $r$ as $r=p^lr'$ for some $l\geq 0$ and a positive integer $r'$ coprime to $p$. 
	Let $c:=\gcd(r',q-1)$ and $k:=\frac{l\sqrt{q}}{\sqrt{q}+1}$.
	\begin{itemize}
		\item[(i)] If $r\mid q$ and $q$ is a square, then $J_r\leq \frac{1}{\sqrt{q}+1} \log_q r$.
		\item[(ii)] If $r\mid(q-1)$, then $J_r\leq 2\log_q r$.
		\item[(iii)] If $r\nmid (q-1)$ and,  $q$ is non-square or $c>p^k$, then $J_r\leq \log_q r$.
		\item[(iv)] If $r\nmid q$, $r\nmid(q-1)$, $q$ is a square, and $c\leq p^k$,  then 
		\[ J_r\leq \frac{l}{\sqrt{q}+1} \log_q p+\log_q (cr').\]
	\end{itemize}	
\end{theorem}

\begin{proof}
	We give a complete proof by comparing the results of \cite{cas:cram:xing}:
	\begin{itemize}
		\item[(i)] Applying \cite[Theorem 2.4(ii)]{cas:cram:xing} with $r=p^l$ and  $r'=c=1$
		we obtain the inequality  
		\[J_r\leq \frac{l}{\sqrt{q}+1} \log_q p.\]
		\item[(ii)]  This assertion is a direct consequence of  \cite[Theorems 2.3 and 2.4]{cas:cram:xing}.
		\item[(iii)] and (iv) When $r \nmid (q-1)$,  \cite[Theorem 2.3(ii)]{cas:cram:xing} yields to $J_r \leq \log_q r$. Furthermore, when $q$ is a square, 
		we obtain 
		\begin{equation}\label{1}
		J_r\leq \frac{l}{\sqrt{q}+1}log_q r,
		\end{equation}
	by \cite[Theorem 2.3(iii)]{cas:cram:xing}.	Using \cite[Theorem 2.4(ii)]{cas:cram:xing}, also the following inequality  holds:
	\begin{equation}\label{2}
	J_r\leq \frac{l}{\sqrt{q}+1}log_q p+log_q(cr').
	\end{equation}
	Hence, by inequalities (\ref{1}), (\ref{2}), and substituting the value $r=p^lr'$, we get 
	\begin{eqnarray*} A &:=&\frac{l}{\sqrt{q}+1}log_q p+log_q(cr')-\log_q r\\
		                &=&\frac{-l\sqrt{q}}{\sqrt{q}+1}\log_q p+\log_q c.
	\end{eqnarray*}
		Since $A\geq 0$ if and only if $c\geq p^k$,  assertion (iv) follows.  
	\end{itemize}
\end{proof}
We remark that  for Theorem \ref{thm_bounds}(iv) with $c<p^k$,  \cite[Theorem 2.4]{cas:cram:xing} gives a better upper bound on  $J_r$ than \cite[Theorem 2.3]{cas:cram:xing}. \\
\begin{remark} \label{rem_1} It is well-known from Weil \cite{weil} that for any function field $F/\F_q$ with genus $g$ one has $|J_F[r]|\leq r^{2g}$, and hence
	Theorem \ref{thm_bounds}(ii) always holds.
\end{remark}
The following definition and theorems will be used in the subsequent sections:
\begin{definition} Let $u\in \mathbb{N}$, $ m_i\in \mathbb{Z}\setminus\{0\}$, and  $Y_i \in Cl(F)$ for $i = 1, \ldots, u$. The Riemann-Roch
	system of equations in the indeterminate $X$ is the system of equations
	 \begin{equation}\label{RRS}
	 \{\ell(m_iX +Y_i) = 0\}_{i=1}^u
	 \end{equation}
	determined by these	data. A solution is some divisor class $[G] \in Cl(F)$ satisfying all equations when substituted for  $X$.
\end{definition}
\begin{theorem}\cite[Theorem 3.2]{cas:cram:xing}\label{thm_cas_1} 
	Consider the Riemann-Roch system (\ref{RRS}). For $i = 1, \ldots, u$ and   $s \in \mathbb{Z}$, let
	\[d_i:= \deg Y_i \quad \textrm{ and } \quad  r_i:= m_is+d_i.\]
	If one has
	\begin{equation*}
		h(F)>\sum_{i=1}^u A_{r_i}\cdot |J_F[m_i]|,
	\end{equation*}	
	then the system (\ref{RRS}) has a solution $[G]\in Cl_s(F)$.
\end{theorem}
\begin{theorem} \cite[Theorem 4.11]{cas:cram:xing} \label{thm_cas_2}  Let $t\geq 1$,  $d\geq 2$. Define $I^*: = \{1,\ldots,n\}$. For $\emptyset \neq A \subset I^*$  define
	\[P_A := \sum_{j\in A} P_j \in Div(F).\]
      Let  further a canonical divisor $K\in Div(F)$ be given. If the system
	\[ \{\ell(dX-D+P_A+Q)=0, \ell(K-X+P_A+Q)=0\}_{A\subset I^*, |A|=t}\]
	is  solvable  for $X$, then there is a solution  $G \in Div(F)$  such that the algebraic-geometry code $C=C(D,G)_L$ is an $(n, t, d, n - t)$-arithmetic
	secret sharing scheme for $\F_q^k$ over $\F_q$ with uniformity.
\end{theorem}

\section{Torsion-limits of towers}\label{torsion}

 We begin with  an application  of Theorem \ref{thm_bounds} when $q$ is a square:

\begin{proposition} Suppose that $q=p^k$ is a square (with $k\geq 1$ and  $p$ prime) and $r=p^lr'$ where $\gcd(r',p)=1$. We set $c:=\gcd(r',q-1)$ and $k:=\frac{l\sqrt{q}}{\sqrt{q}+1}$. Then there exists a recursive tower of function fields $\FF$ over $\F_q$ such that one has 
	 \[A(\FF)\geq \sqrt{q}-1-B+J_r(\FF),\]
	  where 
	\begin{equation*}
		B=\begin{cases}
			\frac{1}{\sqrt{q}+1} \log_q r \qquad \qquad\; \textrm{ if } r\mid q \\
			2\log_q r\qquad\qquad\qquad \textrm{ if } r\nmid q \textrm{ but } r\mid (q-1)\\
			log_q r \qquad\qquad\qquad \quad\textrm{ if } r\nmid q,  r\nmid (q-1), \; c\geq p^k\\
			\frac{l}{\sqrt{q}+1}\log_q p+\log (cr') \quad \textrm{ otherwise. } 
		\end{cases}
	\end{equation*}
\end{proposition}

\begin{proof}
	We know from  \cite{gar:stic} that there exists a recursive tower of function fields $\FF$ over $\F_q$ with $A(\FF)=\sqrt{q}-1$. As $q$ is a square, the proof follows easily from Theorem \ref{thm_bounds}.

\end{proof}
We now need the following result of Bassa et al. \cite{bas:be:gar:stic}:

\begin{theorem}\cite[Theorem 1.2]{bas:be:gar:stic} \label{bassa}  Let $n=2m+1\geq 3$ be an integer and $q=p^n$ with a prime $p$. There exists a recursive tower of function fields $\FF$ over $\F_q$ such that
	\[A(\FF)\geq \frac{2(p^{m+1}-1)}{p+1+\epsilon},\textrm{  where } \epsilon=\frac{p-1}{p^m-1}.\]
\end{theorem}

Next,  the torsion limit of the tower given in Theorem \ref{bassa} can be estimated  by using the lower bound on the Ihara limit $A(\FF)$:

\begin{proposition} \label{prop_1}
	Let $n$ and $q$ be given as in Theorem \ref{bassa}. There exists a recursive tower of function fields $\FF$ over $\F_q$ with the following properties:
	\begin{itemize}
		\item[(i)] If $p$ is odd, then  $A(\FF)\geq A+J_2(\FF)$, where
		\begin{equation}\label{alp}
			A=\frac{2(p^{m+1}-1)}{p+1+\epsilon}-2\log_q 2\textrm{  with } \epsilon=\frac{p-1}{p^m-1}.
		\end{equation}
		\item[(ii)] If $p$ is even, then  $A(\FF)\geq A+\log_q 2+J_2(\FF)$, where $A$ is given as in Eqn. (\ref{alp}).
	\end{itemize}
\end{proposition}

The proof of Proposition \ref{prop_1} is obvious; it follows from Theorems \ref{thm_bounds} and \ref{bassa}, and Remark \ref{rem_1}. 

\section{New conditions for the construction of arithmetic secret sharing schemes}\label{secret}

 For an algebraic function field $F/\F_q$ with genus $g$,  we set
\begin{equation}\label{delta}
	\Delta := \{i \;:\; 1 \leq i \leq  g - 1 \textrm{ and  } B_i \geq 1\} \quad\textrm{ with } \delta:=|\Delta|,
\end{equation}
 fix an integer $n \geq  0$, and further set
\begin{equation}\label{Un}
	U_n:=\{ b = (b_i )_{i\in \Delta}\; : \; b_i \geq  0 \textrm{ and }  \sum_{i\in \Delta} i\cdot b_i = n \}.
\end{equation}
It is well-known that the number of effective divisors of degree $n$ of an algebraic
function field $F/\F_q$ is given as follows:
\[A_n = \sum_{b\in U_n} \bigg[\prod_{i\in \Delta}{B_i+b_i-1\choose b_i}\bigg], \]
see for instance \cite{bal:rol:tut}. By combining this formula for $A_n$ with some  results of \cite{cas:cram:xing} and the bound on class number  given in \cite{la:mar} we obtained the following theorem. This improves the sufficient conditions on the existence of arithmetic secret sharing schemes with uniformity:  

\begin{theorem} \label{thm_1} Let $F/\mathbb{F}_q$ be a function field of genus $g$, $d,k,t,n \in \mathbb{N}$ with $d\geq 2$, $n>1$, and $1 \leq t< n$.  Let  $1\leq m \leq g-1$ be given such that 
 $B_m\geq B_i$ for all $i\in \{1,\ldots,g-1\}$. Moreover, set $f:=\lfloor{\frac{g-1}{m}}\rfloor$. Suppose that $Q_1,Q_2,\ldots,Q_k, P_1,P_2,\ldots, P_n \in \mathbb{P}^{(1)}(F)$ are pairwise distinct rational places and  
	\begin{equation} \label{assumption_2}
	d^{2g}\leq \frac{H-2g\sqrt{q}-q-1}{{B_m+f \choose f}^\delta},
	\end{equation}
	where \[H:=\frac{q^{g-1}\cdot(q-1)^2}{(q+1)\cdot(g+1)}\]
	 and $\delta$ is given as in (\ref{delta}). Assume further that there exists an element $s\in \mathbb{Z}$ such that
	\begin{equation} \label{condition_s}
		2g-s+t+k-2=1\quad \mbox{ and } \quad  1\leq ds-n+t\leq g-1.
	\end{equation}
	
	Then there exists an $(n,t,d,n-t)$-arithmetic secret sharing scheme for $\mathbb{F}_q^k$ over $\mathbb{F}_q$  with uniformity.  
\end{theorem}

\begin{proof} We first note that  $|J_F[d]| \leq d^{2g}$ by Remark \ref{rem_1}. Let $A$ be a subset of $\{1,2,\ldots,n\}$ with $t$ elements, and 
	\[P_A:=\sum_{i\in A} P_i,\; Q:=\sum_{i=1}^k Q_i,\; \textrm{ and }\; D:=Q+\sum_{i=1}^nP_i\]
	 be divisors of $F/\F_q$. Let $K$ be a canonical divisor of $F/\F_q$. Consider the following system of Riemann-Roch equations:
	\begin{equation}\label{system_1}
		\{\ell(dX-D+P_A+Q)=0, \ell(K-X+P_A+Q)=0\}.
	\end{equation}
We apply  Theorems \ref{thm_cas_1} and \ref{thm_cas_2} with 
 \[r_1=2g-s+t+k-2,\;  r_2=ds-n+t, \;  m_1=-1, \; m_2=d, \]
  and an $s\in \mathbb{Z}$ satisfying that $r_i:=m_is+d_i$ for $i=1,2$.  Hence, it is enough to show that   
	\[h=h(F)>A_{r_1}\cdot|\mathcal{J}_F[m_1]|+A_{r_2}\cdot  |\mathcal{J}_F[m_2]|.\]
	This guarantees  that there exists a solution $G\in Div(F)$ of (\ref{system_1}) with  $deg(G)=s$.  Again by Theorems \ref{thm_cas_1} and \ref{thm_cas_2},  this solution yields to an  AG-code $C(G,D)_L \subseteq \mathbb{F}_{q}^k\times \mathbb{F}_q^n$ which is an $(n,t,d,n-t)$-arithmetic secret sharing scheme for $\F_q^k$ over $\F_q$ with uniformity.
	We now set 
	\[H:=\frac{q^{g-1}(q-1)^2}{(q+1)(g+1)}.\]
	It follows from \cite{la:mar} that $h\geq H$.  
	We set $u_j:=|U_{r_j}|$, with $U_{r_j}$ as in (\ref{Un}), for $j=1,2$. Let $m\in \Delta$ such that  
	\begin{equation}\label{max}
	{B_m+\lfloor{\frac{g-1}{m}}\rfloor\choose \lfloor{\frac{g-1}{m}}\rfloor} :=\max \big\{{B_i+\lfloor{\frac{g-1}{i}}\rfloor\choose\lfloor{\frac{g-1}{i}\rfloor}}|\; i \in \Delta\big\}.
	\end{equation}
	Note that  $r_1=1$ implies  $A_{r_1}=A_1=B_1$. We obtain the following inequality by using the bound on $A_{r_2}$ given in  \cite[Theorem 3.5]{bal:rol:tut}: 
	\begin{eqnarray}
		\nonumber A_{r_1}+A_{r_2}\cdot J_F[d] &\leq &B_1+\prod_{i\in \Delta} {B_i+\lfloor{\frac{g-1}{i}}\rfloor\choose\lfloor{\frac{g-1}{i}}\rfloor} \cdot |\mathcal{J}_F[d]|\\
		\nonumber         &\leq & B_1+{B_m+f\choose f}^\delta \cdot |\mathcal{J}_F[d]|\\
		\nonumber         &\leq & B_1+{B_m+f\choose f}^\delta \cdot d^{2g}\\
		&\leq & H. \label{eq_1}
	\end{eqnarray}
	Hence, Inequality (\ref{eq_1}) holds by  the assumption (\ref{assumption_2}) due to  Hasse-Weil bound \cite{stichtenoth}.
\end{proof}

Firstly, we give an  estimatition for $A_n$. Our aim is to estimate the cardinality of $U_n$. We know that the partitions of a number  $n$ is correspond to the set of solutions $(j_1,j_2,...,j_n)$ to the Diophantine equation
\[1j_1+2j_2+3j_3+...+nj_n=n.\]                                                                                 
For example, two distinct partitions of $4$ can be given by $(1, 1, 1, 1), (1, 1, 2)$  corresponding  to the solutions $(j_1,j_2,j_3,j_4)=(4,0,0,0)$, $(2, 1, 0, 0)$, respectively. 
To compute $|U_n|$, we need to find the number  of partitions  $p(n,\delta)$ of $n$ into at most $\delta$ partitions, where $\delta=|\Delta|$. It follows from \cite[p.9]{gupta} that $p(n,\delta) =p_\delta(n+\delta)$, where $p_\delta(n+\delta)$ is defined to be the number of partitions of $n+\delta$ into exactly $\delta$ partitions.  
Each $\delta$ parts must contain at least $1$ item. Thus, it remains $n$ which needs to be distribute into the $\delta$ parts. It is enough to choose how many to put in the first $\delta -1$ parts, since  the number going into the last part is fixed. Hence, there are $\delta-1$ choices, within the range $[0,n]$. This means  we have $n+1$ choices. Therefore,
\begin{equation}\label{partition_1}
	p_\delta(n+\delta)\leq (n+1)^{\delta-1}.
\end{equation}

\begin{theorem} \label{thm_2} Let $F/\mathbb{F}_q$ be a function field, $d,k,t,n \in \mathbb{N}$ with $d\geq 2$, $n>1$, and $1 \leq t< n$.  Let $1\leq m \leq g-1$, be  such that  $B_m\geq B_i$ for all $i\in \{1,\ldots,g-1\}$. Suppose that $Q_1,Q_2,\ldots,Q_k, P_1,P_2,\ldots, P_n \in \mathbb{P}^{(1)}(F)$ are pairwise distinct rational places and  
	\begin{eqnarray} \label{thm_assumption}
		d^{2g}\leq \frac{H-B_1}{(r_2+1)^{\delta-1}\cdot \left(e\cdot \left(1+\frac{r_2-1}{B_m-1}\right)^{r_2}\right)^\delta}, 
	\end{eqnarray} 
	where 
	\[H:=\frac{q^{g-1}\cdot(q-1)^2}{(q+1)\cdot(g+1)}\]
	 and $\delta$ is given as in (\ref{delta}). Assume further that there exists an  $s\in \mathbb{Z}$ such that
	\begin{equation*} \label{condition_s_2}
		2g-s+t+k-2=1\quad \mbox{ and } \quad   ds-n+t\geq 1.
	\end{equation*}
	Then there exists an $(n,t,d,n-t)$-arithmetic secret sharing scheme for $\mathbb{F}_q^k$ over $\mathbb{F}_q$  with uniformity.  
\end{theorem}

\begin{proof} The proof is similar to that  of Theorem \ref{thm_1}. The main difference is that instead of (ref{max}) we the bound  (\ref{binomial_1}) for binomial coefficients.
 Note that $b_i\leq n$ for all $i\in \Delta$. By applying induction on $n$ the following inequality can be proven:
	\begin{eqnarray}\label{binomial_1}
		{B_m+n-1\choose n} &=& {B_m+n-1\choose B_m-1} \\
		\nonumber                   &\leq&  \bigg(\frac{e\cdot (B_m+n-1)}{n}\bigg)^n. 
	\end{eqnarray}
	Hence, by using (\ref{partition_1}) with $n=r_2$  and (\ref{binomial_1}) and definition of $A_n$, we obtain that
	\begin{eqnarray*}
		\nonumber A_{r_1}+A_{r_2}\cdot J_F[d] &\leq &B_1+\sum_{b\in U_{r_2}} \prod _{i\in \Delta} {B_i+b_i-1 \choose B_i-1} \cdot |\mathcal{J}_F[d]|\\
		\nonumber         &\leq & B_1+\sum_{b\in U_{r_2}} {B_m+n-1 \choose B_m-1}^\delta \cdot |\mathcal{J}_F[d]|\\
		\nonumber         &\leq & B_1+(n+1)^{\delta-1}{B_m+n-1\choose B_m-1}^\delta  \cdot d^{2g}\\
		&=& B_1+\frac{[(n+1){B_m+n-1 \choose B_m-1}]^\delta}{n+1} \cdot d^{2g}\\
		&=& B_1+\frac{((n+1)e(1+\frac{n-1}{B_m-1})^n)^\delta}{n+1}\cdot d^{2g}\\
		&\leq & H. \label{eq_1.1}
	\end{eqnarray*}
	This inequality  holds by Assumption (\ref{thm_assumption}).
\end{proof}
\section{Construction of families of schemes with uniformity}\label{uniformity}

 We now  consider exact sequences of function fields over finite fields. The  sufficient conditions on the existence of families of arithmetic secret sharing schemes with uniformity \cite[Theorems 4.15 and 4.16]{cas:cram:xing} can be given  by imposing certain conditions on the sequences of  $\mathcal{F}=\{F_i/\mathbb{F}_q\}_{i\geq 1}$ of function fields. We first need the following results: 

\begin{proposition}\label{h_lim}\cite[Corollary 2]{tsfasman} Let $\FF=\{F_i\}_{i\geq 0}$ be an exact sequence of function fields over a finite field $\F_q$. Then the following limit exists:
	\[h({\FF}):=\lim_{i\to \infty} \frac{\log h_i}{g_i}.\]
\end{proposition}

\begin{theorem} \label{tsfasman_thm}\cite[Theorem 6]{tsfasman}  The following limit exists for an asymptotically exact family of function fields  $\FF$ over any finite field  $\F_q$:
	\[ \Delta(\mu):=\lim_{i\to \infty} \frac{A_{n_i}}{g_i},\]
	where $n_i:= \left\lfloor {\mu g_i} \right\rfloor$ and $\mu\in \mathbb{R}^{\geq 0}$.  Moreover, for
	\begin{equation}\label{mu_0}
	\mu_0:= \sum_{m=1}^\infty \frac{m \beta_m(\FF)}{q^m-1}\; \textrm{ and } \; \mu \geq \mu_0
	\end{equation}
 we have 
	\[\Delta(\mu)=h({\FF})-(1-\mu)\cdot \log q.\]
\end{theorem}
The main result concerning exact sequences of function fields  and  good artihmetic secret sharing schemes  is given with the following theorem.

\begin{theorem} \label{asymp_thm} Let  $d\geq 2$ be a positive integer and  $\mathcal{F}=\{F_i\}_{i\geq 0}$ be an asymptotically exact family of function fields over $\mathbb{F}_q$. Let further  $\mu$ be given as in Condition (\ref{mu_0}).  For any  $n_i, k_i\in \mathbb{N}$, with $i\geq 0$, suppose that  the following assertions hold:
	\begin{itemize}
		\item[(i)] $J_d(\FF)\leq (1-\mu)\log q$,
		\item[(ii)] $B_1(F_i)\geq n_i+k_i$.
	\end{itemize}
	Then there exist $t_i \in \mathbb{N}$ depending on $n_i$ satisfying  $1\leq t_i<n_i$, and  an infinite family of  $\{(n_i,t_i,d,n_i-t_i)\}_{i\geq 0}$ arithmetic secret sharing schemes for $\F_q^{k_i}$ over $\F_q$ with uniformity.
\end{theorem}
\begin{proof}
	For a fixed $i\geq 0$ let
	 $$Q_{i,1}, Q_{i,2},\ldots,Q_{i,k}, P_{i,1}, P_{i,2},\ldots, P_{i,n_i}$$
	be distinct rational places of $F_i/\mathbb{F}_q$. For simplicity we write $k:=k_i$, $n:=n_i$, and $t:=t_i$. Assume that $I=\{1,2,\ldots n_i\}$ and  $A\subseteq I$ with $|A|=t$. 
	Define
	\[P_{i,A}:= \sum_{j\in A} P_{i,j} \in Div(F_i) \textrm{  and }  Q_i:=\sum_{j=1}^{k} Q_{i,j}\in Div(F_i).\]
	Let $K_i\in Div(F_i)$ be a canonical divisor of $F_i/\F_q$ and
	 $$D_i:=Q_i+\sum_{j=1}^n P_{i,j} \in Div(F_i).$$
	  By \cite[Theorem 4.11]{cas:cram:xing} it is enough to show that the system of Riemann-Roch equations 
	\begin{eqnarray}\label{system_2}
	\nonumber \ell(K_i-X+P_{i,A}+Q_i)&=& 0,\\
	  \ell(dX-D_i+P_{i,A}+Q_i)&=& 0  
	\end{eqnarray}
	has a solution $G_i\in Div(F_i)$ such that $\deg G_i=s_i$ so  that the AG-code $C(G_i,D)_L\subseteq \F_q^k\times \F_q^n$ is an $(n,t,d,n-t)$ arithmetic secret sharing scheme for $\F_q^k$ over $\F_q$ with uniformity. We have the property that 
	\[d_{1,i}:=2g_i+t_i+k-2=\deg (K_i+P_{i,A}+Q_i)\; \mbox{ and } \]
	\[d_{2,i}:=t_i-n_i=\deg (-D_i+P_ {i,A}+Q_i).\]
	Notice  that $A_1(F_i)=B_1(F_i)$ and $|J_{F_i}[-1]|=1$. We set  $h_i:=h(F_i)$ for all $i\geq 1$.
	We now  apply Theorems \ref{thm_cas_1} and \ref{thm_cas_2}  with $m_1=-1$, $m_2=d$, and choose $s_i\in \mathbb{Z}$ so that 
	$$r_{1,i}:=m_1+d_{1,i}=2g_i-s_i+t_i+k-2=1\; \textrm{  and }$$
	 $$r_{2,i}:=m_2s_i+d_{2,i}=ds_i-n_i+t_i=\left\lfloor \mu g_i\right\rfloor \geq 1.$$ 
	  This implies  that if 
	\begin{eqnarray}\label{eq_verify_1}
		h_i  \geq  2 A_{r_2}(F_i) |J_{F_i}[d]|> B_1(F_i)+A_{r_2}(F_i) |J_{F_i}[d]|
	\end{eqnarray}
	holds, then the system of equations (\ref{system_1}) has a desired solution $G_i \in Div(F_i)$. To finish the proof, we need to
	verify Inequality (\ref{eq_verify_1}). Taking $\log_q$ of both sides of  (\ref{eq_verify_1}) and dividing them by $g_i$ yield to
	\begin{eqnarray}\label{eq1}
		\frac{\log_q h_i}{g_i}\geq \frac{\log_q 2}{g_i}+ \log_q \frac{A_{r_2}(F_i)}{g_i}+\frac{\log_q |J_{F_i}[d]|}{g_i}.
	\end{eqnarray}
	Since  the sequence  $\mathcal{F}=\{F_i\}_{i\geq 0}$ is  exact, it follows from  Proposition \ref{h_lim} and Theorem \ref{tsfasman_thm} that taking limit infimum of both sides of Inequality (\ref{eq1})  gives that
	\begin{eqnarray}\label{verify}
	\nonumber	h(\FF)&=&\lim_{i\to \infty} \frac{\log_q h_i}{g_i}\\
  \nonumber		   &\geq& \lim_{i\to \infty} \frac{\log_q A_{r_{2,i}(F_i)}}{g_i}+\liminf_{i\to \infty} \frac{\log_q |J_{F_i}[d]|}{g_i}\\
	             	&=&  \Delta(\mu)+J_d(\FF). 
	\end{eqnarray}
	We know from \cite[Proposition 4.1]{tsf:vla} that the following inequality holds:
	\begin{eqnarray}\label{bound_zeta}
     \nonumber		\Delta(\mu)=\liminf_{i\to \infty} \frac{\log A_{r_{2,i}}}{g_i}&=&\mu \log q+\sum_{m=1}^\infty \beta_m \log \frac{q^m}{q^m-1}\\ 
		                                                              &\geq& \mu \log q.
	\end{eqnarray} 
	Now it follows from  Theorem \ref{tsfasman_thm} and Assertion (ii) that  Equation (\ref{verify}) holds, which implies that Inequality
	(\ref{eq_verify_1}) holds for sufficiently large $i$.
\end{proof}

\begin{remark}\label{rem_optimal}
	Suppose that $q$ is a square. Then there are many function field towers $\mathcal{F}=\{F_i\}_{i\geq 0}$ over $\mathbb{F}_q$ with
	 $$\beta_1(\mathcal{F})=A(\mathcal{F})=\sqrt{q}-1\; \textrm{ and }$$
	  $$\beta_i(\mathcal{F})=\lim_{i\to \infty}\frac{B_i(F_i)}{g_i}=0\; \textrm{ for all }\; i\neq 1,$$ 
	  see for instance \cite{gar:stic}. Moreover,  we know from \cite[Corollary 2]{tsfasman} that for any asymptotically exact sequence   $\mathcal{F}=\{F_i\}_{i\geq 0}$ of function fields (which includes towers), the following equality holds:
	\[\lim_{i\to \infty} \frac{\log h_i}{g_i}=\log q+\sum_{i=1}^\infty \beta_i(\mathcal{F}) \log \big(\frac{q^i}{q^i-1}\big).\]
	
\end{remark}
By Remark \ref{rem_optimal} and Theorem \ref{asymp_thm}, we obtain:

\begin{proposition} \label{prop_optimal} Suppose that $q$ is a square and   $d\geq 2$ is a positive integer.
	Let further $\mu$ be given as in Condition (\ref{mu_0}).  There exists a  tower $\mathcal{F}=\{F_i\}_{i\geq 0}$  of function fields over $\mathbb{F}_q$ with $n_i, k_i \in \mathbb{N}$  such that the following conditions hold:
	\begin{itemize}
		\item[(i)] $\mu+J_{d}(\mathcal{F}) \leq \sqrt{q}+(\sqrt{q}-1)\log \big(\frac{q-1}{q}\big)$,
		\item[(ii)] $B_1(F_i)\geq n_i+k_i$ for sufficiently large $i$.
	\end{itemize}
\end{proposition}
An immediate consequence of Proposition \ref{prop_optimal} is the following corollary whose proof follows from Remark \ref{rem_optimal},  and is similar to that of Theorem  \ref{asymp_thm}:
\begin{corollary} Suppose that $q$ is a square and $d,k_i,n_i \in \mathbb{N}$ with $d\geq 2$.
	Then there exist $t_i \in \mathbb{N}$ depending on $n_i$ satisfying  $1\leq t_i<n_i$, and  an infinite family of  $\{(n_i,t_i,d,n_i-t_i)\}_{i\geq 0}$ arithmetic secret sharing schemes for $\F_q^{k_i}$ over $\F_q$ with uniformity.
\end{corollary}

\begin{example} Let $q=\ell^2$, where $\ell$ is a prime power. Consider the tower $\FF=\{F_i\}_{i\geq 0}$ over  $\F_q$ defined by the equation 
\[f(x,y)=y^\ell x^{\ell -1}+y-x^\ell \in \F_q[x,y].\]
This tower is optimal \cite{gar:stic}, i.e. $\beta_1(\FF)=\ell-1$ and  $\beta_i(\FF)=0$ for all $i\geq 2$. Thus, the value of $\mu_0$ defined in Condition (\ref{mu_0})  is
\[\mu_0=\frac{1}{\ell+1}.\]
Choose $\mu=\mu_0$, $d=2$. By\cite[Theorem 2.10]{gar:stic} we have
\begin{equation*}
	g(F_i)=\begin{cases}
		(q + 1)q^{i}-(q + 2)q^{i/2} + 1 \qquad \qquad \quad \textrm{ if $i$ is even}\\
		(q + 1)q^{i}- \frac{1}{2} (q^2 + 3q + 2)q^{({i-1})/2} + 1 \; \textrm{ if $i$ is odd}.
	\end{cases}
\end{equation*}
Moreover, by \cite[Proposition 3.1]{gar:stic}, we have 
\[B_1(F_i)\geq (q-1)\ell^i+2\ell \quad \textrm{ for all $i\geq 4$}.\] 
For each $i\geq 4$ we choose $n_i,k_i \in \mathbb{N}$ in such a way that $B_1\geq n_i+k_i$. Then Proposition \ref{prop_optimal} is satisfied by Theorem \ref{thm_bounds}. Therefore, the tower $\FF$ can be used to construct  an infinite family of  $\{(n_i,t_i,d,n_i-t_i)\}_{i\geq 0}$ arithmetic secret sharing schemes for $\F_{\ell^2}^{k_i}$ over $\F_{\ell^2}$ with uniformity.
\end{example} 
\begin{example}  Consider the example from \cite[Proposition 5.20]{bal:rol:tut}. Let $\FF=\{F_i\}_{i\geq 0}$ be the tower over $\F_9$  defined by the polynomial
\begin{equation}\label{lr.poly}
	f(X,Y)=Y^2+(X+b)^2-1 \in \F_9[X,Y],  b\in \F^*_3
\end{equation}
and $F_0=\F_9(x_0)$ be the rational function field. Let $E=F_0(z)$ with $z$ is a root of the polynomial 
\[ \varphi(T)=(T^2+\alpha^7)(T^9-T)-\frac{1}{x_0} \in F_0[T],\]
where  $\alpha$ is a primitive element for $\F_9$. Then the sequence $\mathcal{E}=\{E_i\}_{i\geq 0}$, 
with $E_i:=EF_i$, over $\F_9$ is a composite quadratic tower such that for all $i\geq 0$,
\begin{itemize}
	\item[(i)] $B_1(E_i)\geq 9\cdot 2^i$ and $B_2(E_i)\geq 2^i $,
	\item[(ii)] 
	\[g_i=\begin{cases} 21\cdot 2^{i-1}-33\cdot 2^{(i-2)/2}+6 \;\textrm{ if $i\equiv 0\bmod 2$,}\\
	21\cdot 2^{i-1}-11\cdot 2^{(i+1)/2}+6 \;\textrm{ if $i\equiv 1\bmod 2$,}
	\end{cases}\]
	where $g_i=g(E_i)$.
	\item[(iii)] $ \beta_1(\mathcal{E})=\frac{6}{7},\; \beta_2(\mathcal{E})=\frac{2}{21}$ and $\beta_j(\mathcal{E})=0$ for all  $j\geq 3$. 
\end{itemize}
Thus, the value of $\mu_0$  defined in Condition (\ref{mu_0}) is
\[\mu_0=\frac{23}{210}\approx 0.12.\]
Choose $\mu=0.5\geq \mu_0$ and for simplicity choose $k_i=2, t_i=50$ for all $i=1,2,3,4$.
From the proof of Theorem \ref{thm_bounds} we obtain  
$$r_{1,i}=1, \; r_{2,1}=2,\; r_{2,2}=7,\; r_{2,3}=23,\; r_{2,4}=54$$
 for all $d=2,3,4,5$. 
We now have the following table by using the relations given in the proof of  Theorem \ref{thm_bounds}:
$$ s_i=2g_i+t+k-3 \;\textrm{ and  }\; n_i=ds_i+t-r_{2,i} \textrm{ for } i=1,2,3,4.$$
Table 1: Some parameters of Example 2
\begin{center}
	\begin{tabular}{|l|c|c|c|c|r|}
		\hline
		$d$        &$j_d \leq$    &$n_1$    &$n_2$  & $n_3$ &$n_4$ \\
		\hline  
		2         & $0.8$        & $166$ &$201$  &$309$ &$525$ \\
		\hline                
		3          & $0.8$        &$225$  &$280$  &$450$ &$791$  \\
		\hline	
		4          &$0.8$         &$288$  &$363$  &$595$ &$1060$  \\
		\hline
		5         & $0.2$        &$343$   &$438$  &$732$  &$1321$ \\
		\hline
	\end{tabular} 
\end{center}
Notice that  for $d=2,3,4,5$ and  $q=9$ we have 
\[J_d(\mathcal{E})\leq \sqrt{q}+(\sqrt{q}-1)\log \big(\frac{q-1}{q}\big)-\mu \approx 1.9. \]

\end{example} 

\section{Conclusion}\label{conclusion}

In this work  some bounds \cite{cas:cram:xing} on  the construction of arithmetic secret sharing schemes are improved by using  bounds on class number \cite{la:mar}. We here estimated the torsion limit of an important class of towers of function fields  \cite{bas:be:gar:stic} depending on the Ihara limit. In the case $d\geq 2$, these new bounds can easily be adapted to improve several applications of torsion limits ranging from improving the communication complexity of zero knowledge protocols  for multiplicative relations  \cite{cramer:damgaard:pastro} and bilinear complexity of finite field multiplication to obtain new results on the asymptotics of frameproof codes.

\section*{Acknowledgment}

We thank Mehmet Sab{\i}r Kiraz for his comments during the preparation of this research. O. Uzunkol is partly supported by a  research project funded by Bundesministerium f\"{u}r Bildung und Forschung (BMBF), Germany (01DL12038) and T\"{U}B\.{I}TAK, Turkey (TBAG-112T011). Uzunkol's research is also partially supported by  the project (114C027) funded by EU FP7-The Marie Curie Action and  T\"UB\.{I}TAK (2236-CO-FUNDED Brain Circulation Scheme).

\bibliographystyle{splncs}

\end{document}